\documentclass{amsart}
\usepackage{amsmath, amscd, amssymb, amsthm}
\usepackage{bbm}
\usepackage{latexsym}
\usepackage{amsfonts}
\usepackage{graphicx}
\usepackage{xcolor}
\usepackage{needspace}
\definecolor{revieworange}{RGB}{224,112,0}
\definecolor{reviewred}{RGB}{180,35,35}
\definecolor{reviewblue}{RGB}{0,82,170}

\usepackage[all,cmtip]{xy}
\usepackage[colorlinks,linkcolor=blue,breaklinks=true,urlcolor=blue,citecolor=blue,anchorcolor=blue,pagebackref]{hyperref}%
\usepackage{geometry}
\newtheorem{theorem}{Theorem}
\newtheorem{lemma}{Lemma}

\newtheorem{proposition}{Proposition}

\newtheorem{question}[theorem]{Question}

\newtheorem{conjecture}{Conjecture}

\newcommand{\HSC}{\operatorname{HSC}}
\newcommand{\Scal}{\operatorname{Scal}}
\newcommand{\Sym}{\operatorname{Sym}}
\newcommand{\Gr}{\operatorname{Gr}}

\renewcommand*\backref[1]{}
\renewcommand*\backrefalt[4]{ \ifcase #1 \or (cited on page #2) \else (cited on pages #2) \fi}

\newcommand{\be}{\begin{equation}}
\newcommand{\ee}{\end{equation}}
\newcommand{\bea}{\begin{eqnarray}}
\newcommand{\eea}{\end{eqnarray}}

\def\XXint#1#2#3{{\setbox0=\hbox{$#1{#2#3}{\int}$ }
\vcenter{\hbox{$#2#3$ }}\kern-.6\wd0}}

\newcommand{\revisioncolor}{\color{purple}}
\newcommand{\rev}[1]{\texorpdfstring{{\revisioncolor #1}}{#1}}

\begin{document}

\title[On compact Hermitian surfaces with positive Strominger-Bismut sectional curvature]{On compact Hermitian surfaces with positive Strominger-Bismut sectional curvature}

\author{Qingsong Wang}
\address{Qingsong Wang. Hal{\i}c{\i}o\u{g}lu Data Science Institute, University of California San Diego, La Jolla, CA 92093, USA}
\email{qswang92@gmail.com}
\thanks{Zheng is the corresponding author. He is partially supported by National Natural Science Foundations of China
with the grant No.12471039 and  12141101, and is supported by the 111 Project D21024.}

\author{Shing-Tung Yau}
\address{Shing-Tung Yau. Yau Mathematical Sciences Center, Tsinghua University, Beijing 100084, China}
\email{{styau@mail.tsinghua.edu.cn}}

\author{Fangyang Zheng}
\address{Fangyang Zheng. School of Mathematical Sciences, Chongqing Normal University, Chongqing 401331, China}
\email{20190045@cqnu.edu.cn; franciszheng@yahoo.com}


\begin{abstract}
 In an earlier work, Yau and Zheng proposed {a Hermitian analogue of a weak form of the Frankel conjecture}, which states that any compact Hermitian manifold with positive Strominger-Bismut sectional curvature must be biholomorphic to the complex projective space. In this article we confirm the conjecture in complex dimension two. This is achieved by a vanishing theorem for anti-self-dual harmonic 2-forms on such surfaces, utilizing a curvature decomposition formula by Ferreira and a refined Kato inequality.
 \end{abstract}

\maketitle

\tableofcontents

\markleft{Qingsong Wang, Shing-Tung Yau, and Fangyang Zheng}
\markright{On compact Hermitian surfaces with positive Strominger-Bismut sectional curvature}

\section{Introduction}\label{intro}

The positivity issue is a central theme in complex geometry. The famous Frankel conjecture states that any compact K\"ahler manifold with positive bisectional curvature must be biholomorphic to ${\mathbb C}{\mathbb P}^n$. The algebraic geometry version is the slightly more general Hartshorne conjecture, which states that any smooth projective variety with ample tangent bundle over an algebraically closed field of characteristic zero must be the projective space over that field. {The} Hartshorne conjecture was solved by Mori in his celebrated paper \cite{Mori} in 1979. Around the same time, Siu and Yau proved {the} Frankel conjecture \cite{SiuYau} in 1980 {using differential-geometric methods}. Utilizing the techniques developed by Mori, Siu-Yau, as well as Hamilton's revolutionary technique on Ricci flow \cite{Hamilton}, Mok \cite{Mok} successfully proved the generalized Frankel conjecture, which states that any compact K\"ahler manifold with non-negative bisectional curvature must be covered by the product of ${\mathbb C}^k$ (with $k\geq 0$) and a compact Hermitian symmetric space. In comparison, the generalized Hartshorne conjecture is still wide open (except in low dimensions). The reduction theorem by {Demailly}, Peternell, and Schneider \cite{DPS} can split off the flat factor, so the conjecture reduces to the statement that any Fano manifold with nef tangent bundle must be a rational homogeneous variety.

A weak form of {the} Frankel conjecture is to replace bisectional curvature by the slightly stronger sectional curvature:  {\em any compact K\"ahler manifold with positive sectional curvature must be biholomorphic to ${\mathbb C}{\mathbb P}^n$.} This is a true statement by the work of Mori and Siu-Yau, but a natural question would be to wonder what happens if one drops the K\"ahlerness in the assumption. This naive generalization fails: there are compact Hermitian manifolds with positive sectional curvature that are not biholomorphic to ${\mathbb C}{\mathbb P}^n$. For instance, in his famous work \cite{Wallach}, Wallach constructed a family of Hermitian metrics on the flag threefold $M^3={\mathbb P}(T_{{\mathbb C}{\mathbb P}^2})$ which has positive sectional curvature everywhere. Nonetheless, one can ask the question:

\begin{question}
What {kinds} of compact complex manifolds can admit a Hermitian metric with positive sectional curvature?
\end{question}

Presumably such manifolds would be  rare and might be classifiable. But at this point, we do not even know the answer in dimension 2, nor do we know if such a manifold must be K\"ahlerian. Note that if such a manifold is K\"ahlerian, then it must be projective (and rationally connected) by the work of Yang \cite{Yang}). In \cite{YZ}, Yau and Zheng raised the following conjecture:

\begin{conjecture}[Yau-Zheng \cite{YZ}]  \label{conj1}
Any compact  Hermitian manifold with positive Strominger-Bismut sectional curvature must be biholomorphic to ${\mathbb C}{\mathbb P}^n$.
\end{conjecture}

Recall that on a given Hermitian manifold $(M^n,g)$, besides the Levi-Civita (Riemannian) connection $\nabla^r$, there are also the Chern connection $\nabla^c$  and the Strominger-Bismut connection $\nabla^b$ (\cite{Strominger}, \cite{Bismut}) that are canonically associated to the metric. When the metric $g$ is K\"ahler, all three connections coincide, but when $g$ is not K\"ahler, these three connections are mutually distinct. As observed in \cite{YZ}, the Levi-Civita (Riemannian) sectional curvature is always greater than or equal to the Strominger-Bismut sectional curvature, so the hypothesis in Conjecture \ref{conj1} is stronger than positive (Riemannian) sectional curvature which was proven to be insufficient for the conclusion by Wallach's example.

The main purpose of this article is to confirm Conjecture \ref{conj1} in dimension two:

\begin{theorem} \label{thm2}
Any compact  Hermitian surface with positive Strominger-Bismut sectional curvature must be biholomorphic to ${\mathbb C}{\mathbb P}^2$.
\end{theorem}

Suppose $(M^2,g)$ is \rev{a} compact Hermitian surface with positive Strominger-Bismut sectional curvature. By the monotonicity theorem of Yau and Zheng mentioned before \cite[Theorem 1.14]{YZ}, we know that $g$ also has positive Levi-Civita sectional curvature. {In other words,} it is positively curved as a Riemannian manifold. The complex structure also gives an orientation. {Thus, by Synge's theorem} \cite{Synge}, we know that $M^2$ is {simply connected}. Hence its first betti number $b_1=0$ is even, thus by a well-known fact (see for instance \cite{BHPV} or \cite{Buchdahl}) we know that $M^2$ is K\"ahlerian. On the other hand, on any Hermitian manifold, the Chern holomorphic sectional curvature is no less than the Strominger-Bismut holomorphic sectional {curvature} by \cite[Theorem {1.15}]{YZ}, so $g$ has positive Chern holomorphic sectional curvature. By Yang's result \cite{Yang}, any K\"ahlerian compact Hermitian manifold with positive Chern holomorphic sectional curvature  must be projective and rationally connected. Therefore we know that the surface $M^2$ in Theorem \ref{thm2} must be a rational surface. Hence to prove the theorem, it suffices to show that $b_2(M)=1$, or equivalently, $b^-_2(M)=0$. So it suffices to prove the following:

\begin{theorem} \label{thm3}
Let $(M^2,g)$ be a compact  Hermitian surface with positive Strominger-Bismut sectional curvature. Then every  harmonic anti-self-dual 2-form on $(M^2,g)$ vanishes. Consequently, $b_2^-(M)=0$.
\end{theorem}

The outline of the proof of Theorem \ref{thm3} goes as follows. In the Hodge-Weitzenb\"ock  formula for the harmonic form, we will use Yau-Zheng's monotonicity formula and Ferreira's skew-torsion decomposition to rewrite the Levi-Civita curvature term, and turn the identity into a scalar differential inequality. {Seaman's refined Kato inequality gives the additional gradient estimate needed.} {By choosing an appropriate power of the norm of the harmonic form and applying a smoothing argument, we obtain a contradiction and thus establish the vanishing theorem.}

\vspace{0.3cm}

\section{Preliminaries}

We begin by fixing the signs and normalizations that will be used in the curvature
comparisons and the Weitzenb\"ock formula.

Let $(M^n,g)$ be a Hermitian manifold. We will denote by $J$ the associated almost complex structure and extend $g=\langle , \rangle$ bilinearly over ${\mathbb C}$. Denote by $\omega$ the K\"ahler form of $g$. The volume form becomes $dV_g=\omega^n/n!$. We write \(\delta\) for the metric codifferential:
\begin{equation}\label{eq:codifferential}
 \delta=-*d*,
\end{equation}
where $*$ is the Hodge star operator. Note that the sign on the right hand side of (\ref{eq:codifferential}) is always {a minus sign} because the real dimension of $M$ is even. The action of \(J\) on a differential form \(\phi\) is
\[
 (J\phi)(X_1,\ldots,X_k)
 =\phi(JX_1,\ldots,JX_k).
\]
For a connection \(\nabla\) on $M$, its torsion and curvature tensors are defined by
\begin{align*}
 T^\nabla(X,Y)
 &=\nabla_XY-\nabla_YX-[X,Y],\\
 R^\nabla(X,Y)Z
 &=\nabla_X\nabla_YZ-\nabla_Y\nabla_XZ-\nabla_{[X,Y]}Z.
\end{align*}
We will write
\[
 R^\nabla(X,Y,Z,W)=\langle R^\nabla(X,Y)Z,W \rangle .
\]
Note that $R^{\nabla}$ is skew-symmetric with respect to its first two positions. When \(\nabla\) is metric, namely, when $\nabla g=0$, then $R^{\nabla}$ is also skew-symmetric with respect to its last two positions, so we can define its sectional curvature by
\begin{equation}\label{eq:sectional-convention}
 K^\nabla(X\wedge Y)
 =\frac{R^\nabla(X,Y,Y,X)}{|X\wedge Y|^2}
 =-\frac{R^\nabla(X,Y,X,Y)}{|X\wedge Y|^2}.
\end{equation}
For \(X\ne0\), the holomorphic sectional curvature is
\[
 \HSC^\nabla(X)=K^\nabla(X\wedge JX).
\]
Equivalently, if \(X\) is real and
\(x=(X-iJX)/\sqrt2\), then
\begin{equation}\label{eq:complex-hsc-convention}
 \HSC^\nabla(X)
 =\frac{R^\nabla(x,\bar x,x,\bar x)}{|x|^4},
 \qquad |x|^2=g(x,\bar x)=|X|^2.
\end{equation}

A connection is \emph{Hermitian} if it preserves \(g\) and \(J\).  Given a Hermitian manifold $(M^n,g)$, there are three canonical metric connections that are widely studied.  The {\em Levi-Civita connection} \(\nabla^r\) is the unique torsion-free metric connection.  The {\em Chern connection} \(\nabla^c\) is the unique Hermitian connection whose torsion has type \((2,0)\), or equivalently,
\[
 T^c(JX,JY)+T^c(X,Y)=0.
\]
The {\em Strominger-Bismut}  {\em connection} \(\nabla^b\)
\cite{Strominger, Bismut} is the unique Hermitian connection for which
\begin{equation}\label{eq:torsion-three-form}
 H(X,Y,Z)=\langle T^b(X,Y),Z \rangle
\end{equation}
is a three-form.  {Here \(T^b\) denotes the vector-valued torsion, while
\(H\in\Omega^3(M;\mathbb R)\) denotes its associated real torsion three-form.}

We record the standard formulas relating the Strominger-Bismut torsion, the fundamental
form, and the Lee form.  The proof fixes the signs used later.

\begin{lemma}\label{lem:bismut-lee}
The Bismut connection and its torsion three-form satisfy
\begin{equation}\label{eq:bismut-explicit}
 H=Jd\omega,
 \qquad
 \nabla^b_XY
 =\nabla^r_XY+\frac12H(X,Y,\mathord\cdot)^\sharp,
\end{equation}
where $\phi^\sharp$ denotes the vector field associated to a one-form $\phi$, namely, $\phi (X) = \langle \phi^\sharp, X\rangle$ for any vector field $X$. {On a Hermitian surface there is a unique real one-form \(\theta\), called the {\em Lee form}, such that}
\begin{equation}\label{eq:lee-definition}
 d\omega=\theta\wedge\omega.
\end{equation}
With the conventions above,
\begin{equation}\label{eq:lee-torsion}
 H=-*\theta,
 \qquad |H|=|\theta|,
 \qquad *dH=\delta\theta,
\end{equation}
\end{lemma}

\begin{proof}
The connection defined by the second equality in  \eqref{eq:bismut-explicit} is metric and has torsion three-form \(H\).  Expanding \(d\omega\) and using the integrability of \(J\)
gives
\[
 2g((\nabla^r_XJ)Y,Z)
 +H(X,JY,Z)+H(X,Y,JZ)=0.
\]
So the connection preserves \(J\), hence is \(\nabla^b\) by its uniqueness
characterization. In complex dimension two, exterior multiplication by \(\omega\) is an
isomorphism from one-forms to three-forms.  This gives the unique form
\(\theta\) in \eqref{eq:lee-definition}.  In a \(J\)-adapted oriented
orthonormal frame,
\[
 J\theta\wedge\omega=-*\theta.
\]
Since \(J\omega=\omega\),
\(H=J(\theta\wedge\omega)=-*\theta\).  The Hodge star is an isometry, and
\(*dH=-*d*\theta=\delta\theta\), which proves
\eqref{eq:lee-torsion}.
\end{proof}

Yau and Zheng \cite[Theorems~1.14 and~1.15]{YZ} proved the following {curvature comparison formulas}, which we will use later:

\begin{proposition}[Yau--Zheng]\label{prop:curvature-comparison}
Let $(M^n,g)$ be a Hermitian manifold. For any real tangent vectors \(X,Y\) spanning a two-plane, we have
\begin{eqnarray}\label{eq:sectional-comparison}
 && \bigl(K^r(X\wedge Y)-K^b(X\wedge Y)\bigr)|X\wedge Y|^2
 \ = \ \frac14|T^b(X,Y)|^2, \\
&& \label{eq:hsc-comparison}
 \HSC^c(X)-\HSC^r(X)
 \ = \ \HSC^r(X)-\HSC^b(X)
 \ = \ \frac{|T^b(X,JX)|^2}{4|X|^4}.
\end{eqnarray}
\end{proposition}
In particular, if a Hermitian manifold $(M^n,g)$ has positive Strominger-Bismut sectional curvature, then its Levi-Civita (Riemannian) sectional curvature $K^r$ and Chern holomorphic sectional curvature $\HSC^c$ will both be positive.

\vspace{0.3cm}

\section{Strominger-Bismut curvature on anti-self-dual two-forms}
\label{sec:curvature-block}

Now suppose that $(M^2,g)$ is a compact Hermitian surface. The complex structure determines an orientation, which gives the Hodge decomposition
\[
 \Lambda^2=\Lambda^+\oplus\Lambda^-.
\]
Let \(\Omega_-^2(M)=\Gamma(M,\Lambda^-)\).  By Hodge
theory, the negative index \(b_2^-(M)\) of the intersection form is the
dimension of the space of harmonic forms in \(\Omega_-^2(M)\). We have
\begin{equation}\label{eq:hermitian-two-form-splitting}
 \Lambda^+=\mathbb R\omega\oplus\Lambda_J^-,
 \qquad
 \Lambda^-=\Lambda^{1,1}_{0,\mathbb R},
\end{equation}
where
\begin{equation*}
 \Lambda_J^-
 =\{\gamma\in\Lambda^2:
 \gamma(J\mathord\cdot,J\mathord\cdot)=-\gamma\}, \ \ \ \ \
  \Lambda^{1,1}_{0,\mathbb R}
 =\{\eta\in\Lambda^2:
 \eta(J\mathord\cdot,J\mathord\cdot)=\eta,
 \ \langle\eta,\omega\rangle=0\} .
\end{equation*}
The latter is the space of real primitive \((1,1)\)-forms.
Note that the superscript in \(\Lambda_J^-\) refers to the action of \(J\), whereas the
superscripts in \(\Lambda^\pm\) refer to the Hodge star.  In particular,
\(\Lambda_J^-\subset\Lambda^+\).

Next let us assume that $g$ has positive Bismut sectional curvature. By the compactness of $M^2$, there exists a uniform constant
\begin{equation}\label{eq:kappa}
 \kappa
 =\min_{\pi\in\Gr_2(TM)}K^b(\pi)>0,
\end{equation}
where $\Gr_2(TM)$ is the bundle of unoriented two-planes in the tangent spaces. With the sign convention in \eqref{eq:sectional-convention}, define the {\em Strominger-Bismut
curvature operator} \(\mathcal R^b\) by
\begin{equation}\label{eq:positive-curvature-operator}
 \langle\mathcal R^b(X\wedge Y),Z\wedge W\rangle
 =-R^b(X,Y,Z,W).
\end{equation}
Then its quadratic form on a unit decomposable two-form is the Strominger-Bismut sectional
curvature.  We remark that the operator need not be self-adjoint because Strominger-Bismut curvature
need not be symmetric under interchange of the first and last pairs.  Set
\begin{equation}\label{eq:Sb-definition}
 S_b=\Sym\bigl(P_-\mathcal R^b|_{\Lambda^-}\bigr),
\end{equation}
where \(P_-\) is orthogonal projection onto \(\Lambda^-\) and
\(\Sym B=(B+B^*)/2\).  Only this symmetric block contributes to the
quadratic estimates below.

\begin{lemma}\label{lem:positive-block}
Let $(M^2,g)$ be a compact Hermitian surface with positive Strominger-Bismut sectional curvature. Then the self-adjoint endomorphism \(S_b\) satisfies
\begin{equation}\label{eq:Sb-lower-bound}
 S_b\ge2\kappa I_{\Lambda^-}.
\end{equation}
\end{lemma}

\begin{proof}
Fix \(x\in M\), a unit form \(v\in\Lambda_x^-\), and a unit
\(\gamma\in(\Lambda_J^-)_x\subset\Lambda_x^+\).  In real four dimensions, a two-form is
decomposable precisely when its self-dual and anti-self-dual parts have equal
norm.  Thus
\[
 \sigma_\pm=\frac{v\pm\gamma}{\sqrt2}
\]
are unit decomposable two-forms.

Since \(\nabla^bJ=0\), the endomorphism \(R^b(X,Y)\) commutes with \(J\).
The two-form obtained from its last two slots is therefore \(J\)-invariant
and orthogonal to \(\Lambda_J^-\).  Equivalently,
\[
 \langle\mathcal R^b\beta,\gamma\rangle=0,
 \qquad \forall \ \beta\in\Lambda_x^2.
\]
Therefore we have
$$ K^b(\sigma_+) = \langle \mathcal R^b\sigma_+,\sigma_+\rangle = \frac12 \langle \mathcal R^b(v+\gamma),v+\gamma \rangle = \frac12 \langle \mathcal R^b(v+\gamma),v \rangle.$$
Expanding $K^b(\sigma_-)$ similarly and adding to the above, we obtain
\[
 K^b(\sigma_+)+K^b(\sigma_-) = \langle \mathcal R^bv,v\rangle = \langle S_bv,v\rangle.
\]
Hence
\[ \langle S_bv,v\rangle \ge2\kappa.
\]
This completes the proof of the lemma.
\end{proof}

Ferreira's four-dimensional skew-torsion decomposition
\cite[{Theorem~3.2}]{Ferreira} has the following consequence for the symmetric
anti-self-dual block.

\begin{proposition}[Ferreira]\label{prop:ferreira-block}
Let $(M^2,g)$ be a Hermitian surface. On \(\Lambda^-\),
\begin{equation}\label{eq:ferreira-block}
 S_b=W^-+
 \left(
  \frac{\Scal_g}{12}
  -\frac{|H|^2}{8}
  +\frac{*dH}{4}
 \right)I_{\Lambda^-},
\end{equation}
where \(W^-\) is the anti-self-dual Weyl curvature of \(g\).
\end{proposition}

\begin{proof}
Ferreira's formula gives the \(\Lambda^-\)-block as
\[
 W^-+
 \left(\frac{\Scal^b}{12}+\frac{*dH}{4}\right)I_{\Lambda^-}
 {-\frac14(d(*H))^-}.
\]
Here \(\Scal^b\) is the scalar curvature of \(\nabla^b\).
{The two-form \((d(*H))^-\) determines, via the cross product on each
oriented Euclidean three-space \(\Lambda_x^-\), a skew-adjoint endomorphism.
Hence this term disappears after taking the symmetric part.}
Ferreira's curvature operator agrees with
\(\mathcal R^b\) in \eqref{eq:positive-curvature-operator}; equivalently, the
lowered curvature tensor in \cite{Ferreira} has the opposite sign from the
convention here.  Finally,
\(\Scal^b=\Scal_g-\tfrac32|H|^2\) by
\cite[formula (2.6)]{Ferreira}.  Substitution gives
\eqref{eq:ferreira-block}.
\end{proof}

In the Hodge--Weitzenb\"ock formula in our later discussion, we will use the following quadratic form associated with \(S_b\).  For \(\xi\in\Lambda^-\), define
\begin{equation}\label{eq:Qb-definition}
 Q_b(\xi)
 =2\bigl((\operatorname{tr}S_b)|\xi|^2
 -\langle S_b\xi,\xi\rangle\bigr).
\end{equation}
Since \(\Lambda^-\) has rank three, \eqref{eq:Sb-lower-bound} gives
\begin{equation}\label{eq:Qb-lower-bound}
 Q_b(\xi)\ge8\kappa|\xi|^2.
\end{equation}
Indeed, if we denote by $a,b,c$ the eigenvalues of $S_b$, then $\frac12 Q_b$ has eigenvalues $b+c$, $c+a$, and $a+b$. Each of them is at least $4\kappa$ by \eqref{eq:Sb-lower-bound}, so \eqref{eq:Qb-lower-bound} holds.

We can now replace the Riemannian curvature term in the Hodge--Weitzenb\"ock
formula by the positive quantity \(Q_b\) and explicit Lee-form terms.

\begin{lemma}\label{lem:pointwise}
Let $(M^2,g)$ be a compact Hermitian surface, and let \(\alpha\in\Omega_-^2(M)\) be harmonic for the Levi-Civita Hodge
Laplacian.  With the function Laplacian
\(\Delta=\operatorname{div}\nabla^r\),
\begin{equation}\label{eq:pointwise-bochner}
 \frac12\Delta|\alpha|^2
 =|\nabla^r\alpha|^2+Q_b(\alpha)
 +\left(\frac{|H|^2}{2}-\delta\theta\right)|\alpha|^2.
\end{equation}
\end{lemma}

\begin{proof}
Taking the trace in \eqref{eq:ferreira-block}, using
\(\operatorname{tr}W^-=0\), and applying \eqref{eq:lee-torsion} gives
\begin{equation}\label{eq:Qb-riemannian-term}
 Q_b(\alpha)
 +\left(\frac{|H|^2}{2}-\delta\theta\right)|\alpha|^2
 =\frac{\Scal_g}{3}|\alpha|^2
 -2\langle W^-\alpha,\alpha\rangle.
\end{equation}
With the curvature conventions above, the Hodge-Weitzenb\"ock formula for a
harmonic anti-self-dual two-form is
\cite[p.~273]{LeBrun}
\[
 0=(\nabla^r)^*\nabla^r\alpha
 +\frac{\Scal_g}{3}\alpha-2W^-\alpha.
\]
Pairing with \(\alpha\) and using
\[
 \langle(\nabla^r)^*\nabla^r\alpha,\alpha\rangle
 =|\nabla^r\alpha|^2-\frac12\Delta|\alpha|^2
\]
proves \eqref{eq:pointwise-bochner}.
\end{proof}

Note that the above pointwise formula still contains the full covariant derivative of
\(\alpha\).  The following sharper form of Kato's inequality relates that term
to the derivative of the scalar norm.  Seaman's refined Kato inequality
\cite[Theorem~1]{Seaman} states that a harmonic self-dual or
anti-self-dual two-form on an oriented Riemannian four-manifold satisfies
\begin{equation}\label{eq:refined-kato}
 |\nabla^r\alpha|^2
 \ge\frac32\,|d|\alpha||^2
\end{equation}
on \(\{\alpha\ne0\}\), and hence almost everywhere because
\(|\alpha|\) is Lipschitz. In the next section, which is the main part of the argument, we will use this inequality to get the vanishing of {anti-self-dual} harmonic two-forms.

\vspace{0.3cm}

\section{Vanishing of anti-self-dual harmonic forms}
\label{sec:vanishing}

Let $(M^2,g)$ be a compact Hermitian surface with positive Strominger-Bismut sectional curvature. Suppose $\alpha$ is a harmonic anti-self-dual two-form. We want to show that $\alpha =0$. Assume the contrary, namely, $\alpha$ is not identically zero. Our goal is to {derive a contradiction}.

The proof will be divided into two steps.  In the first step, we work away from the zero set of $\alpha$  and determine which power of its norm is compatible with both the refined
Kato inequality and the Lee-form terms.  The resulting power is generally not
smooth across the zero set.  Then in the second step we replace it by smooth functions, obtain
estimates independent of the regularization parameter, and pass to the limit.

\subsection{The calculation away from the zero set}

Let $(M^2,g)$ be a compact Hermitian surface with positive Strominger-Bismut sectional curvature and  assume that $\alpha$ is a harmonic anti-self-dual two-form which is not identically zero. Set
\(u=|\alpha|\).  We begin with the scalar inequality that will be used in both
steps.  Equations
\eqref{eq:pointwise-bochner},
\eqref{eq:Qb-lower-bound}, and \eqref{eq:refined-kato} give, on
\(\{u>0\}\),
\begin{equation}\label{eq:scalar-u-inequality}
 u\Delta u
 \ge\frac12|du|^2
 +\left(8\kappa+\frac12|H|^2-\delta\theta\right)u^2.
\end{equation}
For any \(p>0\),
\begin{align}
 \Delta(u^p)
 &=
 p u^{p-1}\Delta u+p(p-1)u^{p-2}|du|^2 \notag\\
 &\ge
 p\left(8\kappa+\frac12|H|^2-\delta\theta\right)u^p
 +p\left(p-\frac12\right)u^{p-2}|du|^2.
 \label{eq:formal-power}
\end{align}
This formula places one restriction on \(p\): its last term is nonnegative
when \(p\ge1/2\).  To see the second restriction, temporarily suppose that
\(u^p\) is smooth on all of \(M\).  For \(p\ge1/2\), we may discard the last
term, multiply by \(u^p\), and integrate by parts.  Using
\(|H|=|\theta|\) and completing the square gives
\[
 0\ge
 \int_M|d(u^p)-pu^p\theta|^2\,dV_g
 +p\int_M
 \left[8\kappa-\left(p-\frac12\right)|H|^2\right]u^{2p}\,dV_g.
\]
For the coefficient of \(|H|^2\) to be nonnegative independently of the
curvature term, one must have \(p\le1/2\).  Thus the unique value at which the
derivative and torsion coefficients are both nonnegative separately is
\[
 p=\frac12.
\]
At \(p=1/2\), the coefficient of the derivative term in
\eqref{eq:formal-power} is zero, and the coefficient of \(|H|^2\) in the
integrated inequality is also zero.  Thus \(u^{1/2}\) is the power for which
the two restrictions agree.

For completeness, substituting \(p=1/2\) into \eqref{eq:formal-power} gives,
on \(\{u>0\}\),
\begin{equation}\label{eq:formal-half-power}
 \Delta(u^{1/2})
 \ge
 \left(4\kappa+\frac14|\theta|^2-\frac12\delta\theta\right)u^{1/2}.
\end{equation}
If \(u^{1/2}\) were smooth, then {multiplying \eqref{eq:formal-half-power} by \(u^{1/2}\), integrating over \(M\), and using the identity}
\(\int_M(\delta\theta)u=\int_M\langle\theta,du\rangle\) {would give}
\begin{equation}\label{eq:formal-half-lee}
 0\ge
 \int_M\left|d(u^{1/2})-\frac12u^{1/2}\theta\right|^2\,dV_g
 +4\kappa\int_Mu\,dV_g.
\end{equation}
Both terms on the right are nonnegative, and the second vanishes only if
\(u=0\).  This is the desired vanishing inequality, but the calculation does
not yet prove the theorem: \(u=|\alpha|\) is only Lipschitz at its zero set, so
\(u^{1/2}\) is not an admissible smooth test function there.

The next step addresses exactly this obstruction.  We take \(p>1/2\), so that
the derivative term in \eqref{eq:formal-power} remains nonnegative, and choose
\(p\) close enough to \(1/2\) that positive curvature controls the remaining
\(|H|^2\)-term.  We then regularize the smooth squared norm \(|\alpha|^2\),
rather than the possibly non-smooth norm \(|\alpha|\).

\subsection{Regularization at the zero set}

Suppose that \(\alpha\in\Omega_-^2(M)\) is a harmonic anti-self-dual form on $(M^2,g)$ that is not identically zero.  Set
\[
 \rho=|\alpha|^2,
 \qquad u=|\alpha|=\sqrt\rho.
\]
The squared norm \(\rho\) is smooth and will be used in the regularization;
the norm \(u\) is used in the refined Kato terms and in the limiting
function.  Unlike \(\rho\), it need not be differentiable on the zero set of
\(\alpha\).

Since \(H\) is bounded, we may choose \(p\in(1/2,1)\) close enough to
\(1/2\) such that
\begin{equation}\label{eq:positive-potential-choice}
 \left(p-\frac12\right)\|H\|_{L^\infty(M)}^2<8\kappa.
\end{equation}
The restriction \(p<1\) will be used when differentiating the regularized
weights below.  Set
\begin{equation}\label{eq:f-epsilon-definition}
 f=u^p=\rho^{p/2},
 \qquad
 f_\varepsilon=(\rho+\varepsilon)^{p/2}
 \quad (\varepsilon>0).
\end{equation}
Thus \(f\) is the nearby power used in the rigorous argument, while
\(f_\varepsilon\) is its smooth approximation and converges pointwise to
\(f\).

The chain rule and Lemma~\ref{lem:pointwise} give, almost everywhere,
\begin{align}
 \Delta f_\varepsilon
 ={}&p(\rho+\varepsilon)^{p/2-1}
 \left[
  |\nabla^r\alpha|^2+Q_b(\alpha)
  +\left(\frac{|H|^2}{2}-\delta\theta\right)\rho
 \right] \notag\\
 &+p(p-2)\rho(\rho+\varepsilon)^{p/2-2}|du|^2.
 \label{eq:chain-rule}
\end{align}
By \eqref{eq:refined-kato}, the two gradient terms on the right are bounded
below by
\begin{equation}\label{eq:gradient-coefficient}
 p(\rho+\varepsilon)^{p/2-2}
 \left[
  \left(p-\frac12\right)\rho+\frac32\varepsilon
 \right]|du|^2,
\end{equation}
which is nonnegative.  Equations \eqref{eq:Qb-lower-bound},
\eqref{eq:chain-rule}, and \eqref{eq:gradient-coefficient} imply
\begin{equation}\label{eq:f-differential-inequality}
 \Delta f_\varepsilon
 \ge p\rho(\rho+\varepsilon)^{p/2-1}
 \left(8\kappa+\frac{|H|^2}{2}-\delta\theta\right).
\end{equation}

After multiplication by \(f_\varepsilon\), the zeroth-order terms have the
common weight
\begin{equation}\label{eq:psi-definition}
 \psi_\varepsilon
 =\rho(\rho+\varepsilon)^{p-1}
 =\frac{\rho}{\rho+\varepsilon}f_\varepsilon^2.
\end{equation}
Since \(M\) is closed,
\(\int f_\varepsilon\Delta f_\varepsilon=-\int|df_\varepsilon|^2\).
Integration of \eqref{eq:f-differential-inequality} therefore gives
\begin{align}
 0\ge{}&
 \int_M|df_\varepsilon|^2\,dV_g
 +8p\kappa\int_M\psi_\varepsilon\,dV_g
 +\frac p2\int_M|H|^2\psi_\varepsilon\,dV_g \notag\\
 &-p\int_M(\delta\theta)\psi_\varepsilon\,dV_g.
 \label{eq:regularized-integral}
\end{align}
The codifferential is the formal adjoint of \(d\), so
\begin{equation}\label{eq:lee-integration-by-parts}
 \int_M(\delta\theta)\psi_\varepsilon\,dV_g
 =\int_M\langle\theta,d\psi_\varepsilon\rangle\,dV_g.
\end{equation}

We next obtain a uniform Sobolev bound for \(f_\varepsilon\) and control the
derivative of \(\psi_\varepsilon\), which is needed for the Lee-form term.
Differentiating \eqref{eq:psi-definition} gives
\begin{equation}\label{eq:psi-derivative}
 d\psi_\varepsilon
 =(\rho+\varepsilon)^{p-2}(p\rho+\varepsilon)\,d\rho.
\end{equation}
Since \(p<1\),
\begin{equation}\label{eq:psi-f-derivative-bound}
 p|d\psi_\varepsilon|
 \le2f_\varepsilon|df_\varepsilon|.
\end{equation}
After dropping the nonnegative zeroth-order terms in
\eqref{eq:regularized-integral}, equations
\eqref{eq:lee-integration-by-parts} and
\eqref{eq:psi-f-derivative-bound}, together with
\(|\theta|=|H|\), give
\begin{equation}\label{eq:H1-bound}
 \|df_\varepsilon\|_{L^2}^2
 \le2\|H\|_{L^\infty(M)}\|f_\varepsilon\|_{L^2}
       \|df_\varepsilon\|_{L^2}.
\end{equation}
For \(0<\varepsilon\le1\), the \(L^2\)-norms of
\(f_\varepsilon\) are uniformly bounded.  Hence
\(\{f_\varepsilon\}\) is bounded in \(W^{1,2}(M)\).

It follows that \(f_\varepsilon\to f\) strongly in \(L^2\).  After passage to a
subsequence, \(f_\varepsilon\rightharpoonup f\) weakly in \(W^{1,2}\), and hence
\(df_\varepsilon\rightharpoonup df\) weakly in \(L^2\).  Moreover,
\begin{equation}\label{eq:psi-pointwise-limit}
 0\le\psi_\varepsilon\le\rho^p,
 \qquad
 \psi_\varepsilon\longrightarrow\rho^p=f^2.
\end{equation}
These bounds give dominated convergence for the zeroth-order terms.  For the
Lee-form term, equation \eqref{eq:psi-derivative} also gives
\begin{equation}\label{eq:psi-dominated-bound}
 |d\psi_\varepsilon|
 \le2u^{2p-1}|du|.
\end{equation}
For \(\rho>0\), this follows from
\[
 (\rho+\varepsilon)^{p-2}(p\rho+\varepsilon)
 \le\rho^{p-1}.
\]
At a zero of the smooth nonnegative function \(\rho\), \(d\rho=0\), so
\eqref{eq:psi-dominated-bound} holds there as well.  The function
\(t\mapsto t^{2p}\) is \(C^1\) on \([0,\infty)\) with derivative zero at the
origin.  The Lipschitz chain rule therefore gives
\(d(f^2)=2p\,u^{2p-1}du\), including almost everywhere on \(\{u=0\}\).
Since \(2p-1>0\), dominated convergence yields
\begin{equation}\label{eq:psi-L1-limit}
 d\psi_\varepsilon\longrightarrow d(f^2)
 \qquad\text{in }L^1.
\end{equation}
We now let \(\varepsilon\to0\) in \eqref{eq:regularized-integral}, using weak
lower semicontinuity for the gradient term, dominated convergence for the
zeroth-order terms, and \eqref{eq:psi-L1-limit} for the Lee-form term.  This
gives
\begin{align}
 0\ge{}&
 \int_M|df|^2\,dV_g
 +8p\kappa\int_Mf^2\,dV_g
 +\frac p2\int_M|H|^2f^2\,dV_g \notag\\
 &-p\int_M\langle\theta,d(f^2)\rangle\,dV_g.
 \label{eq:limit-integral}
\end{align}
The Sobolev chain rule gives \(d(f^2)=2f\,df\).  Completing the square in
\eqref{eq:limit-integral} and using \(|\theta|=|H|\) yields
\begin{equation}\label{eq:completed-square}
 \int_M|df-pf\theta|^2\,dV_g
 +p\int_M
 \left[
  8\kappa-\left(p-\frac12\right)|H|^2
 \right]f^2\,dV_g
 \le0.
\end{equation}
By \eqref{eq:positive-potential-choice}, both terms on the left are
nonnegative, and the second is bounded below by a positive multiple of
\(\int_Mf^2\,dV_g\).  Therefore \(f=0\), {contradicting our assumption} that
\(\alpha\) is not identically zero.

The argument in this section shows that, on a compact Hermitian surface with positive Strominger-Bismut sectional curvature, every harmonic anti-self-dual two-form vanishes, so by Hodge theory we get \(b_2^-(M)=0\). {This completes the proof of Theorem \ref{thm3} and hence, as mentioned in \S 1, that of Theorem \ref{thm2}.}


\vspace{0.2cm}

Finally, we remark that the metric in Theorem \ref{thm2} need not be  K\"ahler. In other words, a non-K\"ahler metric may have positive Strominger-Bismut sectional curvature. For instance, let $g_{FS}$ be the standard Fubini-Study metric on ${\mathbb C}{\mathbb P}^2$ and $f\in C^{\infty }({\mathbb C}{\mathbb P}^2)$ a non-constant function, then for $\varepsilon$ sufficiently small, the non-K\"ahler metric $e^{\varepsilon f}g_{FS}$ would still have positive Strominger-Bismut sectional curvature.

\vspace{0.6cm}

\noindent\textbf{Added in proof.} Very recently, S. Brendle and P. K. Hung \cite{BH} established the existence of a Riemannian metric with positive sectional curvature on $S^2\times S^2$, settling a long-lasting open problem in differential geometry. The main result of this article implies on the other hand that $S^2\times S^2$ cannot admit a Hermitian metric with positive Strominger-Bismut sectional curvature. What remains is the following:

\emph{Can $S^2\times S^2$ admit a Hermitian metric with positive Levi-Civita sectional curvature?}

\vspace{0.6cm}

\noindent\textbf{Acknowledgments.}
The third named author would like to thank Bo Yang, Xiaokui Yang, and Quanting Zhao for their interest and helpful discussions.

\vspace{0.3cm}

\noindent\textbf{Generative AI disclosure.}
The authors used OpenAI's ChatGPT as an assisting tool in the development of the paper. The authors take full responsibility for all mathematical arguments and the manuscript was entirely written by the authors.

\vspace{0.3cm}

\noindent\textbf{Declaration on competing interests.}
All authors declare that there are no competing interests for this paper.

\vspace{0.3cm}

\noindent\textbf{Statement on data availability.}
Data sharing is not applicable to this article as no datasets were generated during the current study.

\end{document}